\documentclass[oneside,english]{amsart}
\usepackage[T1]{fontenc}
\usepackage[latin9]{luainputenc}
\usepackage{amstext}
\usepackage{amsthm}
\usepackage{amssymb}
\usepackage{stmaryrd}
\usepackage{graphicx}

\makeatletter
\numberwithin{equation}{section}
\numberwithin{figure}{section}
  \theoremstyle{definition}
  \newtheorem{defn}{\protect\definitionname}
  \theoremstyle{plain}
  \newtheorem{cor}{\protect\corollaryname}
  \theoremstyle{plain}
  \newtheorem{prop}{\protect\propositionname}
  \theoremstyle{plain}
  \newtheorem{lem}{\protect\lemmaname}
\theoremstyle{plain}
\newtheorem{thm}{\protect\theoremname}
  \theoremstyle{plain}
  \newtheorem*{lem*}{\protect\lemmaname}

\makeatother

\usepackage{babel}
  \providecommand{\definitionname}{Definition}
  \providecommand{\lemmaname}{Lemma}
  \providecommand{\propositionname}{Proposition}
\providecommand{\corollaryname}{Corollary}
\providecommand{\theoremname}{Theorem}

\begin{document}

\title{characteristic classes of involutions in nonsolvable groups}

\author{YOTAM FINE}
\begin{abstract}
Let
$G,D_{0},D_{1}$ be finite groups such that $D_{0}\trianglelefteq D_{1}$
are groups of automorphisms of $G$ that contain the inner automorphisms
of $G$. Assume that $D_{1}/D_{0}$ has a normal $2$-complement and
that $D_{1}$ acts fixed-point-freely on the set of $D_{0}$-conjugacy
classes of involutions of $G$ (i.e., $C_{D_{1}}(a)D_{0}<D_{1}$ for
every involution $a\in G$). We prove that $G$ is solvable. We also
construct a nonsolvable finite group that possesses no characteristic
conjugacy class of nontrivial cyclic subgroups. This shows that an
assumption on the structure of $D_{1}/D_{0}$ above must be made in
order to guarantee the solvability of $G$ and also yields a negative
answer to Problem 3.51 in the Kourovka notebook, posed by A. I. Saksonov
in 1969.
\end{abstract}

\address{Raymond and Beverly Sackler School of Mathematical Sciences, Tel-Aviv
University, Tel-Aviv, Israel}

\email{yotamfine5@gmail.com and also yotamfine@mail.tau.ac.il}

\date{February 27, 2019}
\maketitle

\section{introduction }

Let $G$ be a finite group. A well-studied problem in group theory
is finding interesting conditions on $(G,A)$, where $A$ is a fixed-point-free
group of automorphisms of $G$, that guarantee the solvability of
$G$. It is known {[}3{]} that if $A$ is either cyclic or of order
coprime to $\mid G\mid$, then $G$ is solvable. Of course, $A$ acting
fixed-point-freely on $G$ does not guarantee the solvability of $G$,
as the example $(G,G)$ shows whenever $G$ is centerless and nonsolvable.
One property that the cyclic case and the coprime case share is that
$A$ acts fixed point freely on the set of (nontrivial) conjugacy
classes of $G$. This property clearly does not hold in the $(G,G)$
case. This observation leads to the following two questions, the first
of which was posed by A.I Saksonov in the Kourovka notebook {[}2{]}
in 1969:
\begin{quote}
\textbf{Question 1}\textbf{\textit{.}} (Problem 3.51 of {[}2{]})\textit{
Assume that $G$ possesses no characteristic conjugacy class of nontrivial
elements. Must $G$ be solvable?}\textbf{ }
\end{quote}
\begin{quote}
And in case the answer is negative:
\end{quote}
\begin{quote}
\textbf{Question 2.} \textit{Assume that $A$ is a group of automorphisms
of $G$ that acts fixed-point-freely on the set of nontrivial conjugacy
classes of $G$. Are there any mild conditions on $(G,A)$ that guarantee
the solvabilty of $G$?}
\end{quote}
The main results of this paper are Theorem 1 and Theorem 2. Theorem
2 implies that the answer to Question 1 is negative. Theorem 1 implies
that the answer to Question 2 is positive. Before stating the theorems,
a couple of definitions are needed.
\begin{defn}
A group $D_{0}$ acting (from the left) on a finite group $G$ by
automorphisms is called \textsl{ordinary} iff for every $g\in G$
there exists $\alpha\in D_{0}$ such that $\forall a\in G$ $\alpha(a)=a^{g}$.
An \textsl{ordinary triplet} is a triple of finite groups $(G,D_{0},D_{1})$
such that $D_{1}$ acts on $G$ by automorphisms (from the left),
$D_{0}\trianglelefteq D_{1}$ and $D_{0}$ is ordinary.
\end{defn}
\begin{defn}
An ordinary triplet $(G,D_{0},D_{1})$ is called \textsl{wild }iff
$D_{1}$ acts fixed-point-freely on the set of $D_{0}$-conjugacy
classes of involutions of $G$, i.e., $C_{D_{1}}(a)D_{0}<D_{1}$ for
every involution $a\in G$.
\end{defn}
The first main result of this paper is the following criteria for
solvability.
\begin{flushleft}
\textbf{Theorem 1.} \textit{Assume that $(G,D_{0},D_{1})$ is wild
and that $D_{1}/D_{0}$ has a normal 2-complement. Then $G$ is solvable
and if $B$ is a $D_{1}$-invariant subgroup of $G$, then $(G/B,D_{0},D_{1})$
is wild.}
\par\end{flushleft}

The second main result of this paper is the following theorem.
\begin{flushleft}
\textbf{Theorem 2.} \textit{Let $A$ be a finite group. Then there
exists a finite group $G=H\rtimes A$ ($H$ solvable) such that $G$
possesses no characteristic conjugacy class of nontrivial cyclic subgroups.}
\par\end{flushleft}

We now list some corollaries.
\begin{cor}
Let $G$ be a finite nonsolvable group. Assume that $N\trianglelefteq G$
is a normal subgroup such that $G/N$ has a normal $2$-complement.
Then there exists an involution $a\in N$ such that $C_{G}(a)N=G$.
\end{cor}
\begin{proof}
$N$ is not solvable, and so $(N,N,G)$ is not wild.
\end{proof}
\begin{cor}
Let $G$ be a finite nonsolvable group and $p$ a prime. If every
involution in $G$ is centralized by a Sylow-$p$ subgroup of $G$
(that depends on the involution), then some involution in $G$ is
centralized by a Sylow-$p$ subgroup of $Aut(G)$.
\end{cor}
\begin{proof}
Let $P$ be a Sylow-$p$ subgroup of $Aut(G)$. Looking at $(G,G,G\rtimes P)$,
we see that some conjugacy class of involutions $C$ is $P$-invariant.
From our assumption, $C$ is of order prime to $p$. Thus $P$ centralizes
some element in $C$.
\end{proof}

\section{notation and conventions}

All groups considered in this paper are finite. $^{g}a=a^{(g^{-1})}=gag^{-1}$.
The identity element is denoted by $0$ (yet maintaining multiplicative
notation). The trivial subgroup $\{0\}\leq G$ is identified with
$0$. A written triple in this paper (e.g., ``$(G,D_{0},D_{1})$'')
is automatically assumed to be an ordinary triplet. The action of
the acting groups (in an ordinary triplet) will not be specified as
no confusion will arise. In Section 4, group actions are done from
the right.

\section{lemmas and proof of theorem 1}
\begin{prop}
Assume $(G,D_{0},D_{1})$ is wild. Assume $M\leq G$ and that every
$D_{1}$-conjugate of $M$ is a $D_{0}$-conjugate of $M$ i.e. $N_{D_{1}}(M)D_{0}=D_{1}$.
Then $(M,N_{D_{0}}(M),N_{D_{1}}(M))$ is wild and $N_{D_{1}}(M)/N_{D_{0}}(M)\cong D_{1}/D_{0}$.
\end{prop}
\begin{proof}
Clearly $(M,N_{D_{0}}(M),N_{D_{1}}(M))$ is ordinary and $N_{D_{1}}(M)/N_{D_{0}}(M)\cong D_{1}/D_{0}$.
Let $a\in M$ be an involution. Take $\psi\in D_{1}$ such that $\psi(a)$
is not a $D_{0}$-conjugate of $a$. Take $\alpha\in D_{0}$ such
that $\psi(M)=\alpha(M)$. Now $\alpha^{-1}\psi\in N_{D_{1}}(M)$.
$\alpha^{-1}\psi(a)$ is not a $D_{0}$-conjugate of $a$ and so $\alpha^{-1}\psi(a)$
is not an $N_{D_{0}}(M)$-conjugate of $a$ as well.
\end{proof}
\begin{prop}
Let $(G,D_{0},D_{1})$ be ordinary. Assume $B\leq G$ is $D_{1}$-invariant
(hence also $B$$\trianglelefteq G$). Assume $(G/B,D_{0},D_{1})$
is not wild. Then there exists $B<B_{1}\leq G$ such that $[B_{1}:B]=2$
and $N_{D_{1}}(B_{1})D_{0}=D_{1}$.
\end{prop}
\begin{proof}
Immediate.
\end{proof}
\begin{prop}
Let $(G,D_{0},D_{1})$ be ordinary. Assume $\mid G\mid=2k$ where
$k$ is odd. Then $(G,D_{0},D_{1})$ is not wild.
\end{prop}
\begin{proof}
From Sylow's theorem, $G$ has a unique conjugacy class of involutions.
Taking some involution $a\in G$ (there exists such), every $D_{1}$-conjugate
of $a$ is a $G$-conjugate of $a$ and so it's also a $D_{0}$-conjugate
of $a$.
\end{proof}
\begin{prop}
Assume $(G,D_{0},D_{1})$ is wild. Assume $B\leq G$ is $D_{1}$-invariant
where $\mid B\mid$ is odd. Then $(G/B,D_{0},D_{1})$ is wild.
\end{prop}
\begin{proof}
Assume otherwise. Thus there exists some $B<B_{1}\leq G$ such that
$[B_{1}:B]=2$ and $N_{D_{1}}(B_{1})D_{0}=D_{1}$. Now $(B_{1},N_{D_{0}}(B_{1}),N_{D_{1}}(B_{1}))$
is wild, but $\mid B_{1}\mid=2k$ where $k$ is odd - contradiction.
\end{proof}
\begin{lem}
Let $(M,D_{0},D_{1})$ be ordinary. Assume that $D_{1}/D_{0}$ has
a normal $2$-complement and that $M$ is a $2$-group. Assume some
$B\trianglelefteq M$ with $[M:B]=2$ is $D_{1}$-invariant. Then
$($$M$, $D_{0}$, $D_{1}$$)$ is not wild.
\end{lem}
\begin{proof}
Let $(B,$$M$, $D_{0}$, $D_{1}$$)$ be a counter example with minimal
$\mid M\mid$. 

\textsl{Claim}: No $0<J<B$ is $D_{1}$-invariant.

\textsl{Proof}: Assume otherwise. Let $0<J<B$ be $D_{1}$ invariant.
Passing to $(B/J,M/J,D_{0},D_{1})$ and using the induction hypothesis,
we see that $(M/J,D_{0},D_{1})$ is not wild. Thus there exists some
$J<J_{1}\leq M$ such that $[J_{1}:J]=2$ and $N_{D_{1}}(J_{1})D_{0}=D_{1}$.
Now $(J_{1},N_{D_{0}}(J_{1}),N_{D_{1}}(J_{1}))$ is wild and so, as
$N_{D_{1}}(J_{1})/N_{D_{0}}(J_{1})\cong D_{1}/D_{0}$, we get that
$(J,J_{1},N_{D_{0}}(J_{1}),N_{D_{1}}(J_{1}))$ is also a counter example.
This contradicts the minimality of $\mid M\mid.$ 

$\boxempty$ of the claim.

It follows that $\Omega_{1}(Z(M)\cap B)=B$. Thus (as $B\leq Z(M)$
and $[M:B]=2$) $M$ is abelian. $M$ is elementary abelian as otherwise
we'd have that \rotatebox{90}{\rotatebox{90}{$\Omega$}}$^{1}(M)$
is of order $2$ (Indeed, define a homomorphism $f:M\rightarrow M$
by $f(x)=x^{2}$. Now $B\leq Ker(f)<M$ so $B=Ker(f)$. Thus \rotatebox{90}{\rotatebox{90}{$\Omega$}}$^{1}(M)=Im(f)$
is of order $[M:B]=2$) and obviously fixed by $D_{1}$ - contradicting
the wildness of $(M,D_{0},D_{1})$. Note that as $M$ is abelian,
every group acting on $M$ by automorphims is ordinary.

\textsl{Claim}: No $0<J<B$ satisfies $N_{D_{1}}(J)D_{0}=D_{1}$. 

\textsl{Proof}: Assume otherwise. Let $0<J<B$ satisfy $N_{D_{1}}(J)D_{0}=D_{1}$.
Set $F_{0}=N_{D_{0}}(J)$, $F_{1}=N_{D_{1}}(J)$. Consider $(B,M,F_{0},F_{1})$.
$0<J<B$ is $F_{1}$-invariant. Thus, by the preceding claim applied
to $(B,M,F_{0},F_{1})$, we see that $(B,M,F_{0},F_{1})$ is not a
counter example. As $F_{1}/F_{0}\cong D_{1}/D_{0}$ we get that $(M,F_{0},F_{1})$
is not wild. Now let $a\in M$ be some involution such that every
$F_{1}$-conjugate of $a$ is an $F_{0}$-conjugate of $a$. $(M,D_{0},D_{1})$
is wild and so there is some $\psi\in D_{1}$ for which $\psi(a)$
is not a $D_{0}$-conjugate of $a$. Take $\alpha\in D_{0}$ such
that $\alpha(J)=\psi(J)$. Now $\alpha^{-1}\psi(J)=J$ and so $\alpha^{-1}\psi\in F_{1}$.
Now for some $\beta\in F_{0}$ $\alpha^{-1}\psi(a)=\beta(a)$. Thus
$\psi(a)=\alpha\beta(a)$ but $\alpha\beta\in D_{0}$ - contradiction. 

$\boxempty$ of the claim.

Notice that as $(M,D_{0},D_{1})$ is wild, $(M,C_{D_{1}}(M)D_{0},D_{1})$
is wild. Also, as $D_{1}/D_{0}$ has a normal $2$-complement, $D_{1}/C_{D_{1}}(M)D_{0}$
has a normal $2$-complement, so we may assume that $C_{D_{1}}(M)\leq D_{0}$.
Thus we may in fact assume that $C_{D_{1}}(M)=0$. 

\textsl{Claim}: $D_{0}$ is a $2$-group. 

\textsl{Proof}: Assume otherwise. Let $p\in\pi(D_{0})$ be odd. Take
$P\in Syl_{p}(D_{0})$. $P$ acts on the set $M\setminus B$, in which
it must have a fixed point. Set $J=C_{B}(P)$. It is clear that $N_{D_{1}}(J)D_{0}=D_{1}$.
Now if $J>0$ then $J=B$, but $C_{M}(P)\nleq B$, so $P\leq C_{D_{1}}(M)$
- a contradiction. Thus $J=0$. Now setting $C_{M}(P)=<a>$, we see
that every $D_{1}$-conjugate of $a$ is a $D_{0}$-conjugate of $a$
- a contradiction. 

$\boxempty$ of the claim.

Now $C_{B}(D_{0})>0$ and thus $C_{B}(D_{0})=B$. Note that $D_{1}$
is not a $2$-group (otherwise it would have a fixed point in $M$).
Let $k$ be the size of the normal $2$-complement of $D_{1}/D_{0}$.
Set $\mathcal{F=}\{R\leq D_{1}\mid\,\,\mid R\mid=k\}$. From Schur-Zassenhaus,
$\mathcal{F\neq\emptyset}$ and $D_{0}$ acts transitively on $\mathcal{F}$
by conjugation ($D_{0}$ is possibly trivial). Let $D_{0}<H\leq D_{1}$
be the group for which $H/D_{0}$ is the normal $2$-complement of
$D_{1}/D_{0}$ (so for every $R\in\mathcal{F}$, $D_{0}R=H$). Set
$E=\{C\mid C$ is a $D_{0}$-orbit in $M\setminus B$\}. Let $R\in\mathcal{F}$
be arbitrary. Set $E_{R}$ = \{$C\in E$$\mid$$C$ is $R$-invariant\}.

\textsl{Claim}: Let $C\in E$. Then $C\in E_{R}$ iff some $a\in C$
is fixed by $R$. 

\textsl{Proof}: The ``if'' part is obvious. For the ``only if'':
Assume $C$ is $R$-invariant. Take $b\in C$. Then $C_{\ensuremath{H}}(b)D_{0}=H$
so $C_{H}(b)/C_{D_{0}}(b)\cong H/D_{0}\cong R.$ Now there exists
some $V\leq C_{H}(b)$ with $V\in\mathcal{F}$. For some $t\in D_{0}$,
$R$ $=$ $^{t}V$. Thus $R$ fixes $^{t}b\in C$ and we are done. 

$\boxempty$ of the claim.

\textsl{Claim}: $\mid E_{R}\mid=1$. 

\textsl{Proof}: From Maschke's theorem, we see that $R$ must have
a fixed point in the set $M\setminus B$, so $\mid E_{R}\mid\geq1$.
Now assume $\mid E_{R}\mid>1$. From the previous claim, it follows
that $C_{M}(R)$ contains at least $2$ nontrivial elements. Thus
$C_{M}(R)>C_{M}(R)\cap B>0$. As $D_{0}$ acts transitively on $\mathcal{F}$,
we see that $N_{D_{1}}(C_{M}(R)\cap B)D_{0}=D_{1}$. From a previous
claim, it follows that $C_{M}(R)\cap B=B$. But $C_{M}(R)>C_{M}(R)\cap B$,
and thus we get $C_{M}(R)=M$, which contradicts $C_{D_{1}}(M)=0$. 

$\boxempty$ of the claim.

Write $E_{R}=\{C_{R}\}$. Note that for any $R_{1},R_{2}\in\mathcal{F}$,
we have $C_{R_{1}}=C_{R_{2}}$. Indeed, for some $t\in D_{0}$, $R_{1}$=$^{t}R_{2}$
and thus $C_{R_{1}}=C_{^{t}R_{2}}$$=$$^{t}(C_{R_{2}})=C_{R_{2}}$.
Let $C\in E$ be the unique $D_{0}$-orbit for which for every $R\in$$\mathcal{F}$,
$E_{R}=\{C\}$. It is immediate that $C$ is $D_{1}$-invariant -
a contradiction to the wildness of $($$M$, $D_{0}$, $D_{1}$$)$.
The lemma is proved.
\end{proof}
\begin{lem}
Assume that $(G,D_{0},D_{1})$ is wild and that $D_{1}/D_{0}$ has
a normal $2$-complement. Assume $B$ is a solvable $D_{1}$-invariant
subgroup of $G$. Then $(G/B,D_{0},D_{1})$ is wild.
\end{lem}
\begin{proof}
We may assume that $B$ is a $p$-group for some $p$. The case where
$p$ is odd follows from Proposition 4, so assume $p=2$. Now if $(G/B,D_{0},D_{1})$
is not wild, then there exists some $B<M\leq G$ with $[M:B]=2$ and
$N_{D_{1}}(M)D_{0}=D_{1}$. Thus $(M,N_{D_{0}}(M),N_{D_{1}}(M))$
is wild and $N_{D_{1}}(M)/N_{D_{0}}(M)\cong D_{1}/D_{0}$ and, in
particular, $N_{D_{1}}(M)/N_{D_{0}}(M)$ has a normal $2$-complement.
But $M$ is a $2$-group, $[M:B]=2$ and $B$ is $N_{D_{1}}(M)$-invariant
(as it is $D_{1}$-invariant). This contradicts Lemma 1. The lemma
is proved.
\end{proof}
\begin{lem}
Let $A$ be a simple nonabelian group. Then $A$ (and thus, as $A$
is simple nonabelian, $A^{r}$ for any $r\geq1$) possesses a characteristic
conjugacy class of involutions.
\end{lem}
\begin{proof}
See {[}1{]} (Lemma 12.1).
\end{proof}
Theorem 1 now follows easily:
\begin{thm}
Assume that $(G,D_{0},D_{1})$ is wild and that $D_{1}/D_{0}$ has
a normal 2-complement. Then $G$ is solvable and if $B$ is a $D_{1}$-invariant
subgroup of $G$, then $(G/B,D_{0},D_{1})$ is wild.
\end{thm}
\begin{proof}
It suffices to show that $G$ is solvable. The rest follows from Lemma
2. Let $(G,D_{0},D_{1})$ be a counter example with minimal $\mid G\mid$.
If $G$ possesses a proper nontrivial characteristic subgroup $H$,
then $H$ is solvable and thus $(G/H,D_{0},D_{1})$ is wild and hence
$G/H$ is solvable. Thus $G$ possesses no proper nontrivial characteristic
subgroup. Thus $G\cong A^{r}$ for a simple nonabelian group $A$
and $r\geq1$. This contradicts Lemma 3.
\end{proof}

\section{saksonov's problem}

The purpose of this section is to provide a negative answer to Saksonov's
problem, or more specifically to prove the following theorem:
\begin{thm}
Let $A$ be a finite group. Then there exists a finite group $G=H\rtimes A$
($H$ solvable) such that $G$ possesses no characteristic conjugacy
class of nontrivial cyclic subgroups.
\end{thm}
In order to prove Theorem 2, we need some definitions. A group $G$
is called $<p>$-wild (for a prime $p$) iff there exists no $x\in G$
of order $p$ such that $<x>\preceq G$ (where ``$S\preceq G$''
means that the conjugacy class $\{S^{g}\mid g\in G\}$ is characteristic,
for a subgroup $S\leq G$). Let $\xi(G)=\{p\in\pi(G)\mid G\text{ is \ensuremath{<p>}-wild\ensuremath{\}.}}$
It is evident that ``$G$ possesses no characteristic conjugacy class
of nontrivial cyclic subgroups'' is equivalent to ``$\pi(G)=\xi(G)$''.
It is also evident that if $H$ is a normal $p'$-subgroup of $G$
such that $G/H$ is $<p>$-wild and the natural map $\sigma:N_{Aut(G)}(H)\rightarrow Aut(G/H)$
is surjective, then $G$ is $<p>$-wild. Our strategy is as follows.
Let $A$ be any nontrivial group and $p$ a prime. We shall construct
a semidirect product $G_{p}(A)=B\rtimes A$ where $B$ is a nontrivial
elementary abelian $p$-group such that $G_{p}(A)$ is $<p>$-wild
and the natural map $\sigma:N_{Aut(G_{p}(A))}(B)\rightarrow Aut(G_{p}(A)/B\cong A)$
is surjective. This would yield that $\xi(G_{p}(A))\supseteq\xi(A)\cup\{p\}$.
Theorem 2 would then follow easily: Let $A$ be a (W.L.O.G nontrivial)
finite group. Write $\pi(A)=\{p_{1},..,p_{n}\}.$ Set $G=G_{p_{1}}(G_{p_{2}}...(G_{p_{n}}(A))...)$.
Now $G$ is of the form $H\rtimes A$ ($H$ solvable) and $\pi(G)=\{p_{1},...,p_{n}\}\subseteq\xi(G)$
so $\xi(G)=\pi(G)$ and $G$ is as needed. The rest of this section
is devoted to the construction of $G_{p}(A)$ ($p$ a fixed prime
and $A$ a fixed nontrivial group). The following easy lemma is useful. 
\begin{lem}
Assume $G=B\rtimes A$ and $\psi\in Aut(B)$ satisfies $\psi(v^{g})=\psi(v)^{g}$
for every $v\in B$ and $g\in A$. Then there exists a unique $\psi'\in Aut(G)$
such that $\forall g\in A\,\psi'(g)=g$ and $\forall v\in B\,\psi'(v)=\psi(v)$.
\end{lem}
\begin{proof}
Define $\psi':G\rightarrow G$ via $\psi'(vg)=\psi(v)g$ (where $v\in B$
and $g\in A$). $\psi'$ is clearly a well defined bijection. Also
\[
\psi'(v_{1}g_{1}v_{2}g_{2})=\psi'(v_{1}(v_{2})^{(g_{1}^{-1})}g_{1}g_{2})=\psi(v_{1}(v_{2})^{(g_{1}^{-1})})g_{1}g_{2}=\psi(v_{1})\psi((v_{2})^{(g_{1}^{-1})})g_{1}g_{2}
\]
\[
=\psi(v_{1})\psi(v_{2})^{(g_{1}^{-1})}g_{1}g_{2}=\psi(v_{1})g_{1}\psi(v_{2})g_{2}=\psi'(v_{1}g_{1})\psi'(v_{2}g_{2})
\]

and so $\psi'\in Aut(G)$. Uniqueness is obvious.
\end{proof}
We start the construction of $G_{p}(A)$. Let $r$ be the minimal
prime such that $r\nmid\mid A\mid(p-1)$. Let $B=(\mathbb{Z}/p\mathbb{Z})^{r(\mid A\mid-1)}$.
Write $B=B_{1}\oplus...\oplus B_{r}$ where for each $i$, $B_{i}$
is of dimension $\mid A\mid-1$ with a fixed basis $\{v_{g}^{i}\mid g\in A^{\#}\}$.
We also denote $v_{0}^{1}=...=v_{0}^{r}=0$. Throughout the construction,
the letters $g,h$ and their variants (say $h_{s}$) will denote elements
of $A$ while the letters $t,v$ (and their variants) will denote
elements of $B$. The letter $a$ (and its variants) will denote an
element of $\mathbb{Z}/p\mathbb{Z}$. Both additive and multiplicative
notation will be used for elements of $B$. 
\begin{lem*}
There exists a unique action of $A$ on $B$ (by automorphisms) such
that for all $g_{1},g_{2}\in A$ and $i=1,...,r$: $(v_{g_{1}}^{i})^{g_{2}}=v_{g_{1}g_{2}}^{i}-v_{g_{2}}^{i}$.
\end{lem*}
\begin{proof}
For each $g_{2}\in A$ the sequence $\{v_{g_{1}g_{2}}^{i}-v_{g_{2}}^{i}\}_{g_{1}\in A^{\#},\,i=1,...,r}$
is a basis for $B$. Thus there exists $F_{g_{2}}\in GL(B)$ such
that for every $g_{1}\in A^{\#}$ and $i=1,...,r$ $F_{g_{2}}(v_{g_{1}}^{i})=v_{g_{1}g_{2}}^{i}-v_{g_{2}}^{i}$.
We also have $F_{g_{2}}(v_{0}^{i})=F_{g_{2}}(0)=0=v_{0g_{2}}^{i}-v_{g_{2}}^{i}$
(for $i=1,...,r)$. Clearly $F_{g_{3}}(F_{g_{2}}(v_{g_{1}}^{i}))=F_{g_{2}g_{3}}(v_{g_{1}}^{i})$
for every $g_{1},g_{2},g_{3}\in A$ and $i=1,...,r$. Thus $F_{g_{3}}\circ F_{g_{2}}=F_{g_{2}g_{3}}$.
Uniqueness is obvious.
\end{proof}
Set $G=G_{p}(A)=B\rtimes A$.

\begin{prop}
The natural map $\sigma:N_{Aut(G)}(B)\rightarrow Aut(G/B)$ is surjective.
\end{prop}
\begin{proof}
Let $f\in Aut(G/B).$ Define $F\in Aut(A)$ via $F(g)B=f(gB)$. Now
let $\tilde{F}\in Aut(G)$ be the automorphism induced from $F$ (i.e.,
$\tilde{F}(g)=F(g),\,\tilde{F}(v_{g}^{i})=v_{F(g)}^{i}$). Clearly
$\tilde{F}\in N_{Aut(G)}(B)$ and $\sigma(\tilde{F})=f.$
\end{proof}
Before proving that $G$ is $<p>$-wild, we introduce some automorphisms.
Let $\psi\in GL(B)$ be defined via $\psi(v_{g}^{i})=v_{g}^{i+1}$
for $i=1,..,r-1$ and $\psi(v_{g}^{r})=v_{g}^{1}$. Let $\phi\in GL(B)$
be defined via $\phi(v_{g}^{i})=v_{g}^{i}$ for $i=1,..,r-1$ and
$\phi(v_{g}^{r})=v_{g}^{r}-v_{g}^{1}$. It is easily seen that both
$\psi$ and $\phi$ satisfy the conditions of Lemma 4. We shall use
the same letters to denote their respectable extensions in $Aut(G)$.
Define automorphisms $\psi_{1},...,\psi_{r}:G\longrightarrow G$ via
$\psi_{i}(gv)=gv_{g}^{i}v$. These maps are indeed automorphisms as
\[
\psi_{i}(g_{1}v_{1}g_{2}v_{2})=\psi_{i}(g_{1}g_{2}(v_{1})^{g_{2}}v_{2})=g_{1}g_{2}v_{g_{1}g_{2}}^{i}(v_{1})^{g_{2}}v_{2}
\]
 and 
\[
\psi_{i}(g_{1}v_{1})\psi_{i}(g_{2}v_{2})=g_{1}v_{g_{1}}^{i}v_{1}g_{2}v_{g_{2}}^{i}v_{2}=g_{1}g_{2}(v_{g_{1}}^{i}v_{1})^{g_{2}}v_{g_{2}}^{i}v_{2}
\]
 and 
\[
(v_{g_{1}}^{i}v_{1})^{g_{2}}v_{g_{2}}^{i}v_{2}=(v_{g_{1}}^{i})^{g_{2}}(v_{1})^{g_{2}}v_{g_{2}}^{i}v_{2}=v_{g_{1}g_{2}}^{i}-v_{g_{2}}^{i}+(v_{1})^{g_{2}}+v_{g_{2}}^{i}+v_{2}=v_{g_{1}g_{2}}^{i}+(v_{1})^{g_{2}}+v_{2}=v_{g_{1}g_{2}}^{i}(v_{1})^{g_{2}}v_{2}.
\]

\begin{lem}
Assume $gt\in G$ ($g\in A$,~$t\in B$) is of order $p$ and $g\neq0$.
Then $g$ and $gt$ are $Aut(G)$-conjugates.
\end{lem}
\begin{proof}
For every $k\geq1$ we have $(gt)^{k}=g^{k}t^{(g^{k-1})}t^{(g^{k-2})}\cdot...\cdot t^{g}t$.
Thus $ord(g)=p$ and ${\textstyle {\displaystyle \sum_{k=0}^{p-1}t^{(g^{k})}=0}}$.
Write $t=t_{1}+..+t_{r}$ where $t_{i}\in B_{i}$. Thus also ${\displaystyle \sum_{k=0}^{p-1}t_{i}^{(g^{k})}=0}$
for all $i=1,...,r$. We now focus on $t_{1}$. Write $t_{1}={\displaystyle \sum_{h\in A}a_{h}v_{h}^{1}}$
where $a_{0}=0$. 

\textsl{Claim}: For all $h\in A\setminus<g>$ we have ${\displaystyle \sum_{\xi\in<g>}}a_{h\xi}=0$. 

\textsl{Proof}: For every $k\geq0$ we have 
\[
t_{1}^{(g^{k})}={\displaystyle \sum_{h\in A}a_{h}v_{hg^{k}}^{1}-\sum_{h\in A}a_{h}v_{g^{k}}^{1}=\sum_{h\in A}a_{hg^{-k}}v_{h}^{1}-\sum_{h\in A}a_{h}v_{g^{k}}^{1}}.
\]
Thus:
\[
{\displaystyle 0=\sum_{k=0}^{p-1}t_{1}^{(g^{k})}=\big(\sum_{h\in A}(\sum_{\xi\in<g>}a_{h\xi}})v_{h}^{1}\big)-{\displaystyle \sum_{\xi\in<g>}(\sum_{h\in A}a_{h}})v_{\xi}^{1}
\]

\[
={\displaystyle \big(\sum_{h\in A\setminus<g>}(\sum_{\xi\in<g>}a_{h\xi}})v_{h}^{1}\big)+{\displaystyle \sum_{\xi\in<g>}\big((\sum_{\eta\in<g>}a_{\eta})-\sum_{h\in A}a_{h}}\big)v_{\xi}^{1}
\]

and the claim follows.

$\boxempty$ of the claim.

\textsl{Claim}: Assume $h\in A\setminus<g>$. Then there exists 
\[
c\in[g,B_{1}]\cap(<v_{h\xi}^{1}\mid\xi\in<g>>\oplus<v_{\xi}^{1}\mid\xi\in<g>>)
\]

such that 
\[
t_{1}+c\in<v_{f}^{1}\mid f\in A\setminus\{h\xi\mid\xi\in<g>\}>.
\]

\textsl{Proof}: First, note that if $1\leq k<p$ and $v\in B_{1}$,
then $v^{(g^{k})}-v=u^{g}-u$ where $u={\displaystyle \sum_{l=0}^{k-1}v^{(g^{l})}}$
and so $[g^{k},B_{1}]\leq[g,B_{1}]$. Now set $y{\displaystyle =\sum_{\xi\in<g>}(a_{h\xi}v_{h}^{1})^{\xi}-a_{h\xi}v_{h}^{1}}$.
Thus $y\in[g,B_{1}]$. Also, 
\[
y={\displaystyle \sum_{\xi\in<g>}(a_{h\xi}v_{h}^{1})^{\xi}-a_{h\xi}v_{h}^{1}}=\sum_{\xi\in<g>}(a_{h\xi}v_{h\xi}^{1}-a_{h\xi}v_{\xi}^{1}-a_{h\xi}v_{h}^{1})
\]
 
\[
{\displaystyle =\sum_{\xi\in<g>}a_{h\xi}v_{h\xi}^{1}}-{\displaystyle \sum_{\xi\in<g>}a_{h\xi}v_{\xi}^{1}-\sum_{\xi\in<g>}a_{h\xi}}v_{h}^{1}={\displaystyle \sum_{\xi\in<g>}a_{h\xi}v_{h\xi}^{1}-\sum_{\xi\in<g>}a_{h\xi}v_{\xi}^{1}}.
\]
Now $c=-y$ is as needed.

$\boxempty$ of the claim.

\textsl{Claim}: There exists $z_{1}\in[g,B_{1}]$ such that $t_{1}+z_{1}\in<v_{g}^{1}>$.

\textsl{Proof: }It follows from the previous claim that there exists
$w\in[g,B_{1}]$ such that $t_{1}+w\in<v_{\xi}^{1}\mid\xi\in<g>>$.
Now note that for every $k\geq1$ we have 
\[
(v_{g^{k-1}}^{1})^{g}-v_{g^{k-1}}^{1}=v_{g^{k}}^{1}-v_{g^{k-1}}^{1}-v_{g}^{1}.
\]
This easily implies that there exists $u\in[g,B_{1}]$ such that $t_{1}+w+u\in<v_{g}^{1}>$.
Thus $z_{1}=w+u$ is as needed.

$\boxempty$ of the claim.

We now complete the proof of the lemma. For each $i=1,...,r$ take
$z_{i}\in[g,B_{i}]$ such that $t_{i}z_{i}\in<v_{g}^{i}>$ (the existence
of such $z_{i}$ for $i=1$ was proved above. For arbitrary $i$ the
proof is identical). Set $z=z_{1}\cdot...\cdot z_{r}$. Thus $z\in[g,B]$
and $tz\in<v_{g}^{1},...,v_{g}^{r}>$. Note that if $v_{1},v_{2}\in B$,
then $(v_{1}^{-1})^{g}v_{1}(v_{2}^{-1})^{g}v_{2}=((v_{1}v_{2})^{-1})^{g}v_{1}v_{2}$
and so for every $v\in[g,B]$ there exists $u\in B$ such that $v=(u^{-1})^{g}u.$
In particular, $z=(u^{-1})^{g}u$ for some $u\in B$. Now 
\[
(gt)^{u}=g^{u}t^{u}=g^{u}t=g(u^{-1})^{g}ut=gt(u^{-1})^{g}u=gtz.
\]
Also, it is easily seen that there exists $\varphi\in<\psi_{1},...,\psi_{r}>$
such that $\varphi(gtz)=g$. Now $g=\varphi((gt)^{u})$ and we are
done.
\end{proof}
\begin{prop}
$G$ is $<p>$-wild.
\end{prop}
\begin{proof}
Assume otherwise. Thus there exists $x\in G$ of order $p$ such that
$<x>\preceq G$. 

First assume that $x\in B$. Thus $\psi(<x>)=<x>^{h}$ for some $h\in A$.
As $\psi(h)=h$, we get that for every $k>0$ $\psi^{k}(<x>)=<x>^{(h^{k})}$.
As $ord(\psi)=r$, we get $<x>=\psi^{r}(<x>)=<x>^{(h^{r})}$. As $r\nmid\mid A\mid$,
we get $<x>=<x>^{h}$. Thus we get $\psi(<x>)=<x>^{h}=<x>$. As $r\nmid p-1$,
we get $\psi(x)=x$. Now write $x={\textstyle {\displaystyle \sum_{i=1}^{r}\sum_{g\in A}a_{g}^{i}v_{g}^{i}}}$
where $a_{0}^{1}=...=a_{0}^{r}=0$. As $\psi(x)=x$, we get that for
each $g$, $a_{g}^{1}=...=a_{g}^{r}$. Define $a_{g}=a_{g}^{1}=...=a_{g}^{r}$.
So $x={\displaystyle \sum_{i=1}^{r}\sum_{g\in A}a_{g}v_{g}^{i}}$.
Now $\phi(x)={\displaystyle \sum_{i=2}^{r}}{\displaystyle \sum_{g\in A}}a_{g}v_{g}^{i}$,
so $\phi(<x>)\leq B_{2}\oplus...\oplus B_{r}$, yet no conjugate of
$<x>$ is a subgroup of $B_{2}\oplus...\oplus B_{r}$ - a contradiction. 

Thus $x=gt$ ($g\in A$,~$t\in B$) where $g\neq0$. From Lemma 5,
$g$ and $gt$ are $Aut(G)$-conjugates. In particular, $<g>\preceq G$.
Now $\psi_{1}(<g>)$ is $G$-conjugate to $<g>$. Thus $gv_{g}^{1}=(g^{k})^{hv}$
for some $k\geq1,\,h\in A$ and $v\in B$. We now have 
\[
gv_{g}^{1}=(g^{k})^{hv}=((g^{k})^{h})^{v}=(g^{k})^{h}(v^{-1})^{((g^{k})^{h})}v.
\]
Thus $(g^{k})^{h}=g$, so we also get $v_{g}^{1}=(v^{-1})^{g}v$.
Write $v=v_{1}+...+v_{r}$ where $v_{i}\in B_{i}$. Clearly $v_{g}^{1}=(v_{1}^{-1})^{g}v_{1}$.
Write $v_{1}={\displaystyle \sum_{f\in A}a_{f}v_{f}^{1}}$ where $a_{0}=0$.
Now 
\[
v_{g}^{1}=v_{1}-(v_{1})^{g}={\displaystyle \sum_{f\in A}(a_{f}-a_{fg^{-1}})v_{f}^{1}}+{\displaystyle \sum_{f\in A}a_{f}v_{g}^{1}}
\]
and thus
\[
1=a_{g}-a_{0}+{\displaystyle \sum_{f\in A}a_{f}}={\displaystyle (\sum_{f\in A\setminus\{0,g\}}a_{f}})+2a_{g}.
\]
Also, for every $f\in A\setminus\{0,g\}$ we have $a_{f}=a_{fg^{-1}}$.
It follows that for every $f\in A\setminus<g>$ and $\xi\in<g>$ we
have $a_{f}=a_{f\xi}$. In particular, ${\displaystyle {\displaystyle \sum_{f\in A\setminus<g>}a_{f}=0}}$.
Thus ${\displaystyle 1=(\sum_{f\in<g>\setminus\{0,g\}}a_{f}})+2a_{g}$.
But for every $2\leq k<p$ we have $a_{g^{k}}=a_{g^{k-1}}$ and thus
$a_{g^{k}}=a_{g}$. We now get ${\displaystyle 1=(\sum_{f\in<g>\setminus\{0,g\}}a_{f}})+2a_{g}=(p-2)a_{g}+2a_{g}=0$
- a contradiction. 
\end{proof}


\begin{thebibliography}{1}
\bibitem{key-1} Fried, M. D., Guralnick, R., \& Saxl, J. (1993).
Schur covers and Carlitz\textquoteright s conjecture. Israel journal
of mathematics, 82(1-3), 157-225.

\bibitem{key-10} Mazurov, V. D., \& Khukhro, E. I. (2014). Unsolved
problems in group theory. the kourovka notebook. no. 18. Sobolev Institute
of Mathematics.

\bibitem[3]{key-1} Rowley, P. (1995). Finite groups admitting a fixed-point-free
automorphism group. Journal of Algebra, 174(2), 724.
\end{thebibliography}
\end{document}